\documentclass[12pt]{amsart}

\usepackage[latin1,utf8]{inputenc}
\usepackage[T1]{fontenc} 
\usepackage{textcomp} 
\usepackage[english]{babel}

\usepackage{hyperref} 
\usepackage{graphicx} 
\usepackage{verbatim} 
\usepackage[all]{xy}
\usepackage{float}

\usepackage{lmodern}
\usepackage{mathrsfs}
\usepackage{amsmath,amssymb,amsthm} 
\usepackage{enumitem}

\newtheorem{theorem}{Theorem}[section]
\newtheorem{lemma}[theorem]{Lemma}

\theoremstyle{remark}
\newtheorem*{remark}{Remark}

\newcommand{\et}{\quad\mbox{and}\quad}

\newcommand{\bN}{\mathbb{N}}
\newcommand{\bP}{\mathbb{P}}
\newcommand{\bQ}{\mathbb{Q}}
\newcommand{\bR}{\mathbb{R}}
\newcommand{\bxi}{\boldsymbol{\xi}}
\newcommand{\bZ}{\mathbb{Z}}

\newcommand{\SSS}{S}

\newcommand\ee{\varepsilon}

\newcommand\GrO{\mathcal{O}} 

\newcommand{\LL}{\mathcal{L}_{\bxi}}
\newcommand{\norm}[1]{\|#1\|}

\renewcommand\phi{\varphi}

\renewcommand\theta{\vartheta}

\newcommand{\ux}{\mathbf{x}}
\newcommand{\uy}{\mathbf{y}}
\newcommand{\uz}{\mathbf{z}}

\newcommand{\VV}{{\mathcal{V}}}
\newcommand{\hlambda}{\widehat{\lambda}}

\begin{document}

\baselineskip=17pt 

\title{A transference principle for simultaneous rational approximation}

\author{Ngoc Ai Van Nguyen}
\address{
  University of Information Technology\\
  Vietnam National University, Ho Chi Minh City\\
  Vietnam}
\email{nnaivan@gmail.com}
\author{Anthony Po\"els}
\address{
   D\'epartement de Math\'ematiques\\
   Universit\'e d'Ottawa\\
   150 Louis Pasteur\\
   Ottawa, Ontario K1N 6N5, Canada}
\email{anthony.poels@uottawa.ca}
\author{Damien Roy}
\address{
   D\'epartement de Math\'ematiques\\
   Universit\'e d'Ottawa\\
   150 Louis Pasteur\\
   Ottawa, Ontario K1N 6N5, Canada}
\email{droy@uottawa.ca}
\thanks{Work of the three authors partially supported by NSERC}


\begin{abstract}
\noindent We establish a general transference principle for the irrationality
measure of points with $\bQ$-linearly independent coordinates in $\bR^{n+1}$,
for any given integer $n\geq 1$. On this basis, we recover an important inequality
of Marnat and Moshchevitin which describes the spectrum of the pairs of ordinary
and uniform exponents of rational approximation to those points. For points whose
pair of exponents are close to the boundary in the sense that they almost
realize the equality, we provide additional information about the corresponding
sequence of best rational approximations. We conclude with an application.
\end{abstract}

\maketitle

\section{Introduction}
\label{section: resultat principal}

Let $n$ be a positive integer and let $\bxi=(\xi_0,\dots,\xi_n)$ be
a point of $\bR^{n+1}$ whose coordinates are linearly independent
over $\bQ$. For any integer point $\ux = (x_0,\dots,x_n)\in\bZ^{n+1}$ we set
\[
    L_{\bxi}(\ux) = \max_{1\leq k \leq n}|\xi_0x_k-\xi_kx_0|,
\]
and for each $X\geq 1$ we define
\begin{equation}
    \label{eq: def fonction L(X)}
    \LL(X) = \min\{L_{\bxi}(\ux)\, ; \, \ux\in\bZ^{n+1}\setminus\{0\}, \norm{\ux}\leq X\},
\end{equation}
where $\norm{\cdot}$ denotes the usual Euclidean norm in $\bR^{n+1}$.
The behavior of this irrationality measure $\LL$ is roughly captured
by the quantities
\begin{align}
    \label{eq: def expo}
    \lambda(\bxi) = \sup\{ \lambda\, ; \, \liminf_{X\rightarrow\infty} X^{\lambda}\LL(X) < \infty \}
    \quad \textrm{and}\quad
    \hlambda(\bxi) = \sup\{ \lambda\, ; \, \limsup_{X\rightarrow\infty} X^{\lambda}\LL(X) < \infty \}
\end{align}
which are called respectively the ordinary and the uniform exponents of rational
approximation to $\bxi$. It is well known that they satisfy
\begin{equation}
    \label{eq: partie spectre expo}
    \frac{1}{n} \leq \hlambda(\bxi) \leq 1
    \quad\textrm{and}\quad
    \hlambda(\bxi)\leq \lambda(\bxi) \leq \infty,
\end{equation}
the inequality $\hlambda(\bxi)\geq 1/n$ coming from Dirichlet's box
principle \cite[Theorem 1A, Chapter II]{Schmidt1980}. The study of such
Diophantine exponents goes back to Jarn\'ik \cite{jarnik1938khintchineschen}
and Khinchine \cite{khintchine1926klasse, khintchine1926metrischen}
and remains a topic of much research. Recently Marnat and Moshchevitin
\cite{marnat2018optimal} proved the following inequality conjectured by
Schmidt and Summerer \cite[Section 3, p.~92]{schmidt2013simultaneous}.

\begin{theorem}[Marnat-Moshchevitin]
\label{thm: Marnat-Moshchevitin simultanee}
Let $\bxi\in\bR^{n+1}$ be a point whose coordinates are linearly independent
over $\bQ$. We have
\begin{align}
    \label{eq: Marnat-Moshchevitin approx simultanee reecrite}
    \hlambda(\bxi) + \frac{\hlambda(\bxi)^2}{\lambda(\bxi)}+\dots
      + \frac{\hlambda(\bxi)^{n}}{\lambda(\bxi)^{n-1}}
     \leq 1,
\end{align}
the ratio $\hlambda(\bxi)/\lambda(\bxi)$ being interpreted as $0$ when
$\lambda(\bxi)=\infty$.
\end{theorem}


The formulation given by Marnat and Moshchevitin in \cite{marnat2018optimal}
is slightly different and is complemented by a similar result for the dual
pair of exponents which we omit here.  These authors also show that
\eqref{eq: partie spectre expo}
and \eqref{eq: Marnat-Moshchevitin approx simultanee reecrite} give a complete
description of the set of values taken by $(\lambda,\hlambda)$ at points
$\bxi\in\bR^{n+1}$ with $\bQ$-linearly independent coordinates.  Previous
to \cite{marnat2018optimal}, the problem had been considered by several authors.
The case $n=1$ of Theorem~\ref{thm: Marnat-Moshchevitin simultanee} is
classical, as it reduces to \eqref{eq: partie spectre expo}. The case $n=2$ is a
corollary of the work of Laurent \cite{laurent2006exponents}. The case $n=3$
was established by Moshchevitin in \cite{moshchevitin2012exponents}, and
revisited by Schmidt and Summerer using parametric geometry of numbers
in \cite{schmidt2013simultaneous}. For an alternative proof of the results
of \cite{marnat2018optimal} based only on parametric geometry of numbers
together with partial results towards a more general conjecture, see the PhD thesis of Rivard-Cooke \cite[Chapter~2]{PhDMartin2019}.


Given a subset $\SSS$ of $\bZ^{n+1}$, we define for each $X\geq 1$
\[
    \LL(X;\SSS)
    = \min\{L_{\bxi}(\ux)\; | \; \ux\in\SSS \textrm{ and } 0< \norm{\ux}\leq X\},
\]
with the convention that $\min \emptyset = \infty$. When
$S\nsubseteq \{0\}$, that function is eventually finite and
monotonic decreasing.  Then, upon replacing
$\LL(X)$ by $\LL(X;\SSS)$ in \eqref{eq: def expo} we obtain
two exponents $\lambda(\bxi;\SSS)$, $\hlambda(\bxi;\SSS)$
which satisfy
\begin{equation}
    \label{eq: relation expo avec S et sans S}
    0\leq\hlambda(\bxi;\SSS)\leq \hlambda(\bxi)\leq 1
     \et
    \lambda(\bxi,\SSS)\leq \lambda(\bxi).
\end{equation}
In particular, we have $\lambda(\bxi;\bZ^{n+1})=\lambda(\bxi)$ and
$\hlambda(\bxi;\bZ^{n+1})=\hlambda(\bxi)$.


The next result gives further information about the behaviour of $\LL(X;\SSS)$
as a function of $X$.

\begin{theorem}
    \label{thm: transfert exposants avec constante}
Let $\bxi\in\bR^{n+1}$ with $\bQ$-linearly independent coordinates
and let $\SSS\subseteq\bZ^{n+1}$. Suppose that there exist positive real numbers $a,b,\alpha,\beta$ such that
\begin{equation}
  \label{eq: encadrement initial version exposants thm1.2}
  bX^{-\beta} \leq \LL(X;\SSS) \leq aX^{-\alpha}
\end{equation}
for each large enough real number $X$. Then $\alpha$ and $\beta$ satisfy
     \begin{equation}
        \label{eq: thm 2 egalite ratio exposants}
        \alpha + \frac{\alpha^2}{\beta}+\dots+\frac{\alpha^n}{\beta^{n-1}} \leq 1.
    \end{equation}
In case of equality in \eqref{eq: thm 2 egalite ratio exposants}, we have
    \begin{equation}
        \label{eq: thm 1.2 estimation des constantes}
        \limsup_{X\rightarrow\infty}X^{\alpha}\LL(X;\SSS) > 0
        \et
        \liminf_{X\rightarrow\infty}X^{\beta}\LL(X;\SSS) < \infty,
    \end{equation}
thus $\alpha = \hlambda(\bxi;\SSS)$ and $\beta=\lambda(\bxi;\SSS)$.
\end{theorem}

Assuming that $\hlambda(\bxi;\SSS) > 0$, the first part of
Theorem~\ref{thm: transfert exposants avec constante} implies that
\begin{align}
    \label{eq: Marnat-Moshchevitin relatif a S}
    \hlambda(\bxi;\SSS) + \frac{\hlambda(\bxi;\SSS)^2}{\lambda(\bxi;\SSS)}+\dots + \frac{\hlambda(\bxi;\SSS)^{n}}{\lambda(\bxi;\SSS)^{n-1}} \leq 1,
\end{align}
which gives Theorem~\ref{thm: Marnat-Moshchevitin simultanee} by
choosing $\SSS = \bZ^{n+1}$. Indeed, if $\lambda(\bxi;\SSS)<\infty$, then \eqref{eq: encadrement initial version exposants thm1.2} holds for $X$ large enough with $a=b=1$ and any choice of $\alpha, \beta$
with $0<\alpha < \hlambda(\bxi;\SSS)$ and $\beta>\lambda(\bxi;\SSS)$. Inequality \eqref{eq: thm 2 egalite ratio exposants} then gives \eqref{eq: Marnat-Moshchevitin relatif a S} by letting $\alpha$ tend to $\hlambda(\bxi;\SSS)$ and $\beta$ to $\lambda(\bxi;\SSS)$. Otherwise, we have $\lambda(\bxi;S) = \infty$ and \eqref{eq: Marnat-Moshchevitin relatif a S} holds trivially since $\hlambda(\bxi;\SSS)\leq 1$. Another application of Theorem~\ref{thm: transfert exposants avec constante} is given in Section~\ref{Section: applications}.

Rather than taking monomials to control the function $\LL$, we now turn to a more general setting in the spirit of \cite{jarnik1938khintchineschen}. The following transference principle is our main result.  As we will see, it implies Theorem~\ref{thm: transfert exposants avec constante}.

\begin{theorem}
\label{thm: thm transfert}
Let $\bxi\in\bR^{n+1}$ with $\bQ$-linearly independent coordinates and
let $\SSS\subseteq\bZ^{n+1}$. Suppose that there exist an unbounded subinterval $I$ of $(0,\infty)$, a point $A\in I$ and continuous functions $\phi,\psi,\theta:I\rightarrow (0,\infty)$ with the following properties.
\begin{enumerate}[label=\rm(\roman*)]
    \item \label{condition 4 thm ppal}  We have $\psi(X)\leq\LL(X;S) \leq \phi(X)$ for each $X\geq A$.
    \smallskip
    \item \label{condition 1 thm ppal}  The functions $\phi$ and $\psi$
      are strictly decreasing, whereas $\theta$ is increasing with
    \[
        \lim_{X\rightarrow\infty} \phi(X) = \lim_{X\rightarrow\infty} \psi(X)= 0
        \quad \textrm{and} \quad
        \lim_{X\rightarrow\infty} \theta(X) = \infty.
    \]
    \item \label{condition 1 bis thm ppal} For each $k=1,\dots,n-1$, the $k$-th iterate $\theta^k$ of $\theta$ maps $[A,\infty)$ to $I$.
    \smallskip
    \item \label{condition 2 thm ppal}  We have $\phi(X) = \psi(\theta(X))$ for each $X\geq A$.
    \smallskip
    \item \label{condition 3 thm ppal}  The functions $\phi_0,\dots,\phi_{n-1}$, $\Phi_0,\dots,\Phi_{n-1}$ defined on $[A,\infty)$ by
    \begin{align}
        \label{eq: def phi_0}
        \phi_0(X) &= \phi(X) \\
        \label{eq: def phi_k et Phi_k}
        \phi_k(X) &= \phi(\theta^{k}(X))\cdots \phi(\theta(X))\phi(X)\quad(1\leq k < n), \\
        \label{eq: def Phi_k}
        \Phi_k(X) &= X\phi_k(X)\quad(0\leq k\leq n-1).
    \end{align}
    have the property that $\Phi_0$ is monotonically increasing and that $\Phi_1,\dots,\Phi_{n-1}$ are monotonic (either decreasing or increasing).
\end{enumerate}
  Then $\Phi_0,\dots,\Phi_{n-2}$ are monotonically increasing and we have
\begin{equation}
    \label{thm: eq ppale thm transfert}
    \Phi_{n-1} \geq c,
\end{equation}
for some constant $c>0$ depending only on $\bxi$.
\end{theorem}

Note that since $\phi$ is decreasing and $\theta$ is increasing, each function $\phi_k$ is decreasing and tends to $0$. The most natural choice for the functions
$\phi,\phi,\theta$ is to take monomials in $X$ as below. In doing so, we now prove
that Theorem~\ref{thm: thm transfert} implies
Theorem~\ref{thm: transfert exposants avec constante}. With the notation of
Theorem~\ref{thm: transfert exposants avec constante}, the functions
$\psi,\phi,\theta$ defined for each $X>0$ by
\begin{equation}
    \label{eq: def psi, phi, theta for power functions}
    (\psi,\phi,\theta)(X) = \Big(bX^{-\beta},aX^{-\alpha},\Big(\frac{a}{b}\Big)^{-1/\beta}X^{\alpha/\beta} \Big)
\end{equation}
satisfy $\phi = \psi\circ\theta$. Moreover since $\alpha\leq \hlambda(\bxi,\SSS)\leq 1$ by \eqref{eq: relation expo avec S et sans S}, the product $\Phi_0(X) = X\phi(X) = aX^{1-\alpha}$ is monotonically increasing for $X>0$. For each $k$ with $0\leq k\leq n-1$ and $X>0$ we have
\begin{equation}
\label{eq: eq inter triplet compatibles functions inter}
    \phi(\theta^k(X))
    = a\Big(\frac{a}{b}\Big)^{(\alpha/\beta) +\dots + (\alpha/\beta)^k}X^{-\alpha^{k+1}/\beta^k},
\end{equation}
and so the functions $\Phi_1,\dots,\Phi_{n-1}$ defined by \eqref{eq: def Phi_k} are monotonic. Thus $\phi,\psi,\theta$ satisfy Conditions \ref{condition 1 thm ppal} to \ref{condition 3 thm ppal} of Theorem~\ref{thm: thm transfert}, and Condition~\ref{condition 4 thm ppal} amounts to Condition \eqref{eq: encadrement initial version exposants thm1.2} of Theorem~\ref{thm: transfert exposants avec constante}. Furthermore note that there is a positive number $\delta>0$ (which is a polynomial in $\alpha/\beta$) such that for each $X>0$ we have
\begin{equation}
    \label{eq: def Phi_n-1 generale pour corollaires}
    \Phi_{n-1}(X) = a^n\Big(\frac{a}{b}\Big)^\delta X^{\ee},\quad\textrm{where } \ee = 1-\Big(\alpha+\frac{\alpha^2}{\beta}+\dots + \frac{\alpha^n}{\beta^{n-1}}\Big).
\end{equation}
By \eqref{thm: eq ppale thm transfert} we then get \eqref{eq: thm 2 egalite ratio exposants}, namely $\ee\geq 0$. This in turn implies \eqref{eq: Marnat-Moshchevitin relatif a S} as explained after Theorem~\ref{thm: transfert exposants avec constante}. Suppose now that $\ee = 0$. Since $\alpha\leq \hlambda(\bxi;\SSS)$ and $\beta\geq\lambda(\bxi;\SSS)$, we thus have
\begin{align*}
    1 = \alpha + \frac{\alpha^2}{\beta}+\dots+\frac{\alpha^n}{\beta^{n-1}} \leq \hlambda(\bxi;\SSS) + \frac{\hlambda(\bxi;\SSS)^2}{\lambda(\bxi;\SSS)}+\dots + \frac{\hlambda(\bxi;\SSS)^{n}}{\lambda(\bxi;\SSS)^{n-1}} \leq 1,
  \end{align*}
and we conclude that $\alpha = \hlambda(\bxi;\SSS)$ and
$\beta=\lambda(\bxi;\SSS)$. Moreover, by using once again
\eqref{thm: eq ppale thm transfert},
\eqref{eq: def Phi_n-1 generale pour corollaires} and the
hypothesis that $\ee=0$, we obtain
\[
    a^n\Big(\frac{a}{b}\Big)^\delta \geq c,
\]
where $c$ is given by \eqref{thm: eq ppale thm transfert}. It means that in \eqref{eq: encadrement initial version exposants thm1.2}, we cannot replace $a$ by a constant strictly smaller that $a'=(cb^\delta)^{1/(n+\delta)}$ and $b$ by a constant strictly larger than $b'=(a^{n+\delta}/c)^{1/\delta}$. This proves \eqref{eq: thm 1.2 estimation des constantes} with the superior limit $\geq a'$ and the inferior limit $\leq b'$.

\begin{remark}
Clearly, Conditions \ref{condition 1 thm ppal} to \ref{condition 3 thm ppal}
apply to many more general classes of functions $\phi$ and $\psi$. For example,
we can take $\phi(X) = aX^{-\alpha}\log^{\sigma}(X)$ and
$\psi(X) = bX^{-\beta}\log^{\rho}(X)$ for suitable positive numbers
$a,b,\alpha,\beta$ and real numbers $\sigma,\rho$.
\end{remark}

The next result complements Theorem~\ref{thm: transfert exposants avec constante}.

\begin{theorem}
\label{thm: comportement points minimaux cas extremal}
Let $n>1$, let $\bxi$ be a point of \/$\bR^{n+1}$ whose coordinates are linearly independent
over $\bQ$ and let $\SSS\subseteq\bZ^{n+1}$. Suppose that there are positive
real numbers $a,b,\alpha,\beta$ such that
\begin{equation}
\label{eq: encadrement initial version exposants}
    bX^{-\beta} \leq \LL(X ; \SSS) \leq aX^{-\alpha}
\end{equation}
for each sufficiently large real number $X$. Then we have $\alpha\leq \beta$ and
\begin{equation}
    \label{eq: egalite expo dans thm points min}
    \ee:= 1-\Big(\alpha + \frac{\alpha^2}{\beta}+\dots + \frac{\alpha^{n}}{\beta^{n-1}}\Big) \geq 0.
\end{equation}
Moreover, there exists a constant $C>0$ which depends only
on $\bxi,a,b,\alpha,\beta$ with the following property. If
\begin{equation}
\label{eq: preuve thm 2.1 def ee_0}
    \ee \leq \frac{1}{4n}\Big(\frac{\alpha}{\beta}\Big)^n
             \min\{\alpha,\beta-\alpha\},
\end{equation}
then there is an unbounded sequence $(\uy_i)_{i\geq 0}$ of non-zero integer
points in $\SSS$ which for each $i\geq 0$ satisfies the following conditions:
\begin{enumerate}[label=\rm(\roman*)]
    \item \label{condition i thm suite}
    $ \big|\alpha\log \norm{\uy_{i+1}}-\beta\log\norm{\uy_i}\big|
      \leq C +4\ee(\beta/\alpha)^n\log \norm{\uy_{i+1}}$;
    \smallskip
    \item \label{condition ii thm suite}
    $ \big|\log L_{\bxi}(\uy_i) + \beta\log \norm{\uy_{i}}\big|
      \leq C+4\ee(\beta/\alpha)^2\log \norm{\uy_i}$;
    \smallskip
    \item \label{condition iii thm suite}
    $\det(\uy_i,\dots,\uy_{i+n}) \neq 0$;
    \smallskip
    \item \label{condition iv thm suite}
    there exists no $\ux\in \SSS\setminus\{0\}$ with  $\norm{\ux} \le \norm{\uy_i}$
    and  $L_{\bxi}(\ux) < L_{\bxi}(\uy_i)$.
\end{enumerate}
\end{theorem}

For a point $\bxi$ of the form $\bxi = (1,\xi,\xi^2)$ with $\xi\in\bR$
not algebraic of degree at most $2$ over $\bQ$, satisfying
\eqref{eq: encadrement initial version exposants} with
$S=\bZ^3$, $\beta=1$  and $\ee=0$, we recover a construction of the
third author \cite[Theorem~5.1]{roy2004approximation} dealing with
extremal numbers. For a point $\bxi=(1,\theta,\dots,\theta^n,\xi)$ with
$\theta\in\bR$ algebraic of degree $n$ over $\bQ$ and
$\xi\in\bR\setminus\bQ(\theta)$, satisfying
\eqref{eq: encadrement initial version exposants} with
$S=\bZ^{n+1}$, $\beta = 1/(n-1)$ and $\ee=0$, the result is due to
the first author \cite[Theorem~2.4.3]{N2014}.

\begin{remark}
As the proof will show, the upper bound for $\ee$ in
\eqref{eq: preuve thm 2.1 def ee_0} and the coefficients of $\ee$
in \ref{condition i thm suite} and \ref{condition ii thm suite}
can easily be improved.
\end{remark}

This paper is organized as follows. In Section~\ref{section: notations}
we set the notation and we recall the definition of minimal points.
Section~\ref{Section:construction de Nguyen} is devoted to our main tool
which is a construction of subspaces of $\bR^{n+1}$ defined
over $\bQ$, together with inequalities relating their heights. The proofs
of Theorems~\ref{thm: thm transfert}
and~\ref{thm: comportement points minimaux cas extremal} follow in
Sections~\ref{section: preuve du theoreme principal}
and~\ref{section: preuves thm comportement points min}
respectively. Finally, some applications of our results are presented
in the last section.

\section{Notation, heights and minimal points}
\label{section: notations}

Given points $\uy_1,\uy_2,\dots$ of $\bR^{n+1}$, we denote by
$\langle \uy_{1}, \uy_2, \dots \rangle_{\bR}$  the vector subspace
of $\bR^{n+1}$ that they span. Recall that we endow $\bR^{n+1}$ with its
usual structure of inner product space and that we denote by  $\norm{\cdot}$
the corresponding Euclidean norm. In general, for any integer $k$
with $1\leq k \leq n+1$,
we endow the vector space $\bigwedge^k(\bR^{n+1})$ with the unique
structure of inner product space such that, for any orthonormal basis
$(e_1,\dots,e_n)$ of $\bR^{n+1}$, the products
$e_{i_1}\wedge\dots\wedge e_{i_k}$ ($i_1<\dots <i_k$) form an
orthonormal basis of $\bigwedge^k(\bR^{n+1})$. We still denote by
$\norm{\cdot}$ the associated norm.

If $W$ is a subspace of $\bR^{n+1}$ defined over $\bQ$, we define its \textsl{height}
$H(W)$ as the co-volume in $W$ of the lattice of integer points $W\cap \bZ^{n+1}$.
If $\dim W = k$, this is given by
\[
    H(W) = \norm{\ux_1\wedge\dots\wedge\ux_k}
\]
for any $\bZ$-basis $(\ux_1,\dots,\ux_k)$ of $W\cap \bZ^{n+1}$.
Schmidt proved the following result \cite[Chap. 1, Lemma 8A]{Schmidt1991}.

\begin{theorem}[Schmidt]
\label{thm: Schmidt inequality}
There exists a positive constant $c$ which depends only on $n$ such that
for any subspaces $A,B$ of \/ $\bR^{n+1}$ defined over $\bQ$, we have
\begin{equation*}
    \label{eq: Schmidt inequality}
    H(A+B)\, H(A\cap B) \leq c H(A)\, H(B).
\end{equation*}
\end{theorem}

If $f,g:I\rightarrow[0,+\infty)$ are two functions on a set $I$,
we write $f=\GrO(g)$ or $f\ll g$ or $g \gg f$ to mean that there
is a positive constant $c$ such that $f(x)\leq cg(x)$
for each $x\in I$. We write $f\asymp g$ when both $f\ll g$
and $g \ll f$.

When $\SSS\subseteq\bZ^{n+1}$ is such that $\lim_{X\rightarrow\infty}\LL(X;S) = 0$,
there exists a sequence $(\ux_i)_{i\geq 0}$ of non-zero points in
$\SSS$ satisfying
\begin{enumerate}[label=(\alph*)]
    \item $\norm{\ux_0} < \norm{\ux_1} < \norm{\ux_2 } < \dots$
    \item $L_{\bxi}(\ux_0) > L_{\bxi}(\ux_1) > L_{\bxi}(\ux_2) > \dots$
    \item For any $i\geq0$ and any non-zero point $\uz\in\SSS$
       with $\norm{\uz} < \norm{\ux_{i+1}}$, we have $L_{\bxi}(\uz) \geq L_{\bxi}(\ux_i)$.
\end{enumerate}
We say that such a sequence is a sequence of \textsl{minimal points} for $\bxi$ with
respect to $\SSS$.
Minimal points are a standard tool for studying rational approximation.
The usual choice is to take $\SSS=\bZ^{n+1}$.

\section{Families of vector subspaces}
\label{Section:construction de Nguyen}

The goal of this section is to prove the following key-theorem
established by the first author in her thesis \cite[\S 2.3]{N2014}
in the case where $\SSS=\bZ^{n+1}$. The proof in the general case
is the same. In this section $n$ is an integer $>1$.

\begin{theorem}
\label{thm: lemme Nguyen}
Let $\bxi\in\bR^{n+1}$ with $\bQ$-linearly independent coordinates.
Suppose that for some $\SSS\subseteq\bZ^{n+1}$ we have
$\lim_{X\rightarrow\infty}\LL(X;\SSS) = 0$. Let $(\ux_i)_{i\geq 0}$ be
a sequence of minimal points for $\bxi$ with respect to $\SSS$. For each $i\geq 0$, set
\[
     X_i = \norm{\ux_i}
     \quad\text{and}\quad
     L_i=\LL(X_i;\SSS) = L_{\bxi}(\ux_i).
\]
Fix also an index $i_0\geq 0$. Then for each $t=1,\dots,n-1$ there exists
a largest integer $i_t$ with $i_t\ge i_0$ such that
\begin{equation} \label{equation2.19}
		\dim \langle \ux_{i_0}, \ux_{i_0+1}, \ldots, \ux_{i_t} \rangle_{\bR} = t+1.
\end{equation}
For these indices $i_0<i_1<\cdots<i_{n-1}$, we have
\begin{equation*}
        X_{i_1}\cdots X_{i_{n-1}} \leq c L_{i_0}X_{i_0+1}\cdots L_{i_{n-1}}X_{i_{n-1}+1}
\end{equation*}
with a constant $c>0$ depending only on $\bxi$ and not on $i_0$.
\end{theorem}

We first note that under the conditions of Theorem~\ref{thm: lemme Nguyen},
each subsequence $(\uy_i)_{i\in \bN}$ of $(\ux_i)_{i\in \bN}$ spans $\bR^{n+1}$.
Indeed, suppose by contradiction that a subsequence $(\uy_i)_{i\in \bN}$
spans a proper subspace $W$ of $\bR^{n+1}$. Since $(\uy_i)_{i \in \bN}$
converges to $\bxi$ projectively, we deduce that $\bxi \in W$, which
is impossible since $W$ is defined by linear equations with coefficients
in $\bQ$ while the coordinates of $\bxi$ are linearly independent over $\bQ$.
In particular, $(\ux_i)_{i\ge i_0}$ spans $\bR^{n+1}$ for the given index
$i_0$, and the existence of $i_1,\dots,i_{n-1}$ follows.

Clearly we have $i_0 < i_1 < \ldots < i_{n-1}$.  For simplicity,
we set
\[
        \VV[i,j] := \langle \ux_{i}, \ux_{i+1}, \ldots, \ux_{j} \rangle_{\bR}
\]
for each pair of integers $i,j$ with $0\leq i \leq j$. Then, for each
$t=0,1,\dots,n-1$, we have
\[
 \dim \VV[i_0,i_t]=t+1
  \quad\text{and}\quad
 \dim \VV[i_0,i_t+1]=t+2,
\]
thus $\ux_{i_t+1}\notin \VV[i_0,i_t]$.  By comparing  dimensions, we deduce that
\begin{equation}
  \label{eq: thmN: eq1}
  \bR^{n+1}= \VV[i_0,i_{n-1}+1]
    \quad\text{and}\quad
  \VV[i_0,i_{t-1}+1] = \VV[i_0,i_t] \quad (1\le t \le n-1).
\end{equation}

For each $(t,k) \in \bN^2$ with $1 \le k \le t+1\le n$, we define
\[
 V^{k+1}_t  = \VV[s(t,k), i_t +1]
 \quad \textrm{and} \quad
 U^k_t  = \VV[s(t,k), i_t],
\]
where $s(t,k)$ is the largest integer with $s(t,k)\le i_t$ such that
$\dim V^{k+1}_t = k+1$. By varying $k$ for fixed $t$, we obtain
a decreasing sequence
\[
  s(t,1) =i_t > s(t,2) >\ldots >s(t,t+1) \ge i_0.
\]
Thus $U^k_t$ is contained in $\VV[i_0,i_t]$ , so $\ux_{i_t+1}
\notin U^k_t$ and therefore $\dim U^k_t = k$. Moreover,
when $2\le k\le t+1$, we have $s(t,k)<s(t,k-1)\le i_t$, thus
\begin{equation}
  \label{eq10}
  V^{k+1}_t = U^k_t + V^k_t
\end{equation}
is the sum of two distinct $k$-dimensional subspaces. Since
$U^k_t$ and $V^k_t$ both contain $U^{k-1}_t$, we deduce that
\begin{equation}
  \label{equation2.11}	
  U^{k-1}_t = U^k_t \cap V^k_t,
\end{equation}
as both sides have dimension $k-1$.  Finally, we note that, for
$t=1,\dots,n-1$, the subspaces $U^{t+1}_t$ and $V^{t+1}_{t-1}$
are both contained in $\VV[i_0, i_{t}] = \VV[i_0, i_{t-1}+1]$.
Since all of these have dimension $t+1$, we conclude that
\begin{equation}
  \label{eq15}
  U^{t+1}_t = \VV[i_0, i_{t-1}+1] = V^{t+1}_{t-1}.
\end{equation}
The proof of Theorem \ref{thm: lemme Nguyen} relies on the
following lemma relating the heights of the above families of
subspaces.

\begin{lemma}
    \label{lem: lemme Nguyen hauteurs}
For each $k=1,\dots,n-1$, we have
\begin{equation}
  \label{equation7}
  H(U^k_k) H(U^k_{k+1}) \cdots H(U^k_{n-1})
   \ll H(V^{k+1}_{k-1}) H(V^{k+1}_k) \cdots H(V^{k+1}_{n-1})
\end{equation}
with an implicit constant depending only on $n$.
\end{lemma}

\begin{proof}
We proceed by descending induction on $k$. By \eqref{eq: thmN: eq1} and
\eqref{eq10}, we have
\[
 \bR^{n+1} = V^{n+1}_{n-1}= U^n_{n-1} + V^n_{n-1}.
\]
Since \eqref{equation2.11} gives $U^{n-1}_{n-1}
= U^n_{n-1} \cap V^n_{n-1}$, it follows from Schmidt's Theorem
\ref{eq: Schmidt inequality} that
\[
 H(U^{n-1}_{n-1}) \ll H(U^n_{n-1}) H(V^n_{n-1})
\]
because $H(\bR^{n+1})=1$.	As \eqref{eq15} gives $H(U^n_{n-1})
= H(V^n_{n-2})$, this proves \eqref{equation7} for $k = n-1$.

Assume that \eqref{equation7} holds for some $k$ with $1 <k \le n-1$.
By Theorem~\ref{eq: Schmidt inequality}, the relations \eqref{eq10} and
\eqref{equation2.11} imply that
\[
 H(V^{k+1}_t) \ll \frac{H(U^k_t) H(V^k_t)}{H(U^{k-1}_t)}
\]
for each $t = k-1,\ldots, n-1$. Combining this with the
induction hypothesis, we obtain
\[
 H(U^k_k) \cdots H(U^k_{n-1})
   \ll \frac{H(U^k_{k-1}) H(V^k_{k-1})}{H(U^{k-1}_{k-1})}
       \cdots
       \frac{H(U^k_{n-1}) H(V^k_{n-1})}{H(U^{k-1}_{n-1})}.
\]
After simplification, this leads to
\[
 H(U^{k-1}_{k-1}) \cdots H(U^{k-1}_{n-1})
   \ll H(U^k_{k-1}) H(V^k_{k-1}) \cdots H(V^k_{n-1}).
\]
Since $U^k_{k-1} = V^k_{k-2}$ by \eqref{eq15}, this yields
\eqref{equation7} with $k$ replaced by $k-1$. Thus, by
induction, \eqref{equation7} holds for all $k =1, \ldots,n-1$.
\end{proof}


\begin{proof}[Proof of Theorem~\ref{thm: lemme Nguyen}]
By Lemma~\ref{lem: lemme Nguyen hauteurs} applied with $k=1$,
we have
\[
 H(U^1_1) H(U^1_{2}) \cdots H(U^1_{n-1})
  \ll H(V^{2}_{0}) H(V^{2}_1) \cdots H(V^{2}_{n-1})
\]
where $U^1_t = \langle \ux_{i_t} \rangle_{\bR}$
and $V^2_t = \langle \ux_{i_t}, \ux_{i_t+1}  \rangle_{\bR} $
for $t =0, \ldots, n-1$.  The conclusion follows since
\[
 H(U_{t}^1) = \| \ux_{i_t} \| = X_{i_t}
 \quad\text{and}\quad
 H(V^2_{t}) \leq \| \ux_{i_t}  \wedge \ux_{i_t+1}
     \| \ll  X_{i_t+1} L_{i_t}
\]
for $t =0, \ldots, n-1$, with implicit constants depending only on $\bxi$.
\end{proof}

%
%

\section{Proof of Theorem~\ref{thm: thm transfert}}
\label{section: preuve du theoreme principal}

Suppose that $\bxi\in\bR^{n+1}$, $\SSS\subseteq\bR^{n+1}$, $A\in I$
and $\phi,\psi,\theta: I \rightarrow (0,\infty)$ satisfy
the hypotheses of Theorem~\ref{thm: thm transfert}, and
let $(\ux_i)_{i\geq 0}$ be a sequence of minimal points
for $\bxi$ with respect to $\SSS$. Since $\Phi_0$ is
monotonically increasing, the case $n=1$ of
Theorem~\ref{thm: thm transfert} is trivial. Thus we may suppose
that $n>1$. As in Section~\ref{Section:construction de Nguyen},
we write $X_i = \norm{\ux_i}$ and $L_i=L_{\bxi}(\ux_i) = \LL(X_i;S)$
for each $i\geq 0$. Choose $k_0\geq 0$ such that $X_{k_0} \geq A$.
Then, for each $i\geq k_0$ and $\ee\in (0,1]$ we have
\[
 \psi(X_i)\leq L_i
   = \LL(X_i;\SSS)
   = \LL(X_{i+1}-\ee;\SSS)
   \leq \phi(X_{i+1}-\ee)
\]
by definition of minimal points.  Letting $\ee$ tend to $0$, we
deduce that
\begin{equation}
\label{eq: X_i minore par X_i+1}
    L_i \leq \phi(X_{i+1})
    \et
    X_i \geq \theta(X_{i+1})\quad (i\geq k_0),
\end{equation}
because $\phi=\psi\circ\theta$ is continuous and $\psi$ is strictly
decreasing.  Then, for each $i_0\ge k_0$, the sequence of integers
$i_0<\dots<i_{n-1}$ given by Theorem~\ref{thm: lemme Nguyen}
satisfies
\begin{equation}
    \label{eq: inegalite de Nguyen}
        X_{i_1}\cdots X_{i_{n-1}}
        \leq c \Phi_0(X_{i_0+1})\cdots \Phi_0(X_{i_{n-1}+1}),
    \end{equation}
where $c=c(\bxi)>0$ and $\Phi_0(X) = X\phi(X)$ as in \eqref{eq: def Phi_k}.

\begin{lemma}
 \label{lem: lemme intermediaire}
Suppose that the functions $\Phi_0,\dots,\Phi_{m-2}$ are
monotonically increasing for some integer $m$ with $2\le m \le n$,
and let $j_0,\dots,j_{m-1}$ be integers with
$k_0\leq j_0<\dots < j_{m-1}$.  Then we have
\begin{equation}
  \label{eq: conclusion lemme 2}
  \Phi_0(X_{j_0+1})\cdots \Phi_0(X_{j_{m-1}+1})
    \leq  X_{j_1}\cdots X_{j_{m-1}}\Phi_{m-1}(X_{j_{m-1}+1}).
\end{equation}
\end{lemma}

\begin{proof}
    For simplicity set $Y_k = X_{j_k}$ and $Z_k= X_{j_k+1}$ for $k=0,\dots, m-1$.
    By induction on $k$, we show that
    \begin{equation}
        \label{eq: H.R. pour lemme 2}
        \prod_{\ell=0}^{m-1}\Phi_0(Z_\ell) \leq \Big(\prod_{\ell=1}^{k-1}Y_\ell\Big)\Phi_{k-1}(Z_{k-1}) \Big(\prod_{\ell=k}^{m-1} \Phi_0(Z_\ell)\Big) \quad(k=1,\dots, m).
    \end{equation}
    The case $k=1$ is an equality; there is nothing to prove. Suppose
    that \eqref{eq: H.R. pour lemme 2} holds for some $k$
    with $1\leq k < m$. We have $Z_{k-1}\leq Y_k$ since $j_{k-1} < j_k$
    and $\theta(Z_k) \leq Y_k$ by \eqref{eq: X_i minore par X_i+1}. Since
    $\Phi_{k-1}$ is monotonically increasing and $\phi_{k-1}$ is
    monotonically decreasing, we deduce that
    \begin{align}
    \label{eq: preuve lemme tranfert fonctionnel}
        \Phi_{k-1}(Z_{k-1})
          \leq \Phi_{k-1}(Y_{k})
          = Y_k\phi_{k-1}(Y_k)
          \leq Y_k\phi_{k-1}(\theta(Z_k)).
    \end{align}
    Since $\phi_{k-1}(\theta(Z_k))\Phi_0(Z_k) = \Phi_{k}(Z_k)$, we
    conclude that \eqref{eq: H.R. pour lemme 2} holds as well with $k$
    replaced by $k+1$. The inequality \eqref{eq: conclusion lemme 2}
    corresponds to the case $k=m$.
\end{proof}

\begin{lemma}
    \label{lem: lemme intermediaire 2}
    The functions $\Phi_0,\dots,\Phi_{n-2}$ are monotonically increasing.
\end{lemma}

\begin{proof}
    Otherwise there is a largest integer $m$ with $2\leq m < n$
    such that $\Phi_0,\dots,\Phi_{m-2}$ are monotonically increasing.
    By our choice of $m$, the function $\Phi_{m-1}$ is monotonically
    decreasing. It is thus bounded from above. Let
    $i_0< i_1 < \dots < i_{n-1}$ be integers satisfying
    \eqref{eq: inegalite de Nguyen} for a choice of $i_0\geq k_0$.
    For simplicity we write $Y_k=X_{i_k}$ and $Z_k=X_{i_k+1}$
    ($k=0,\dots, n-1$).
    Then, Lemma~\ref{lem: lemme intermediaire} applied to $j_0=i_{n-m},\dots, j_{m-1} = i_{n-1}$ implies that
    \[
        \Phi_0(Z_{n-m})\cdots \Phi_0(Z_{n-1}) \leq Y_{n-m+1}\cdots Y_{n-1}\Phi_{m-1}(Z_{n-1}) = \GrO(Y_{n-m+1}\cdots Y_{n-1}),
    \]
    with an implicit constant depending only on $\Phi_{m-1}$,
    not on $i_0$. Furthermore for $k=0,\dots, n-m-1$ we have
    $\Phi_0(Z_k) = \phi(Z_k)Z_k \leq \phi(Z_k)Y_{k+1} = o(Y_{k+1})$
    as $i_0$ tends to infinity. Putting these inequalities together yields
    \begin{align*}
        \Phi_0(Z_{0})\cdots \Phi_0(Z_{n-1}) = o(Y_1\cdots Y_{n-1})
    \end{align*}
    as $i_0$ tends to infinity. This contradicts \eqref{eq: inegalite de Nguyen}.
\end{proof}

\begin{proof}[Proof of Theorem~\ref{thm: thm transfert}]
     Fix $i_0<\cdots<i_{n-1}$ satisfying \eqref{eq: inegalite de Nguyen} for some $i_0\geq k_0$. According to Lemma~\ref{lem: lemme intermediaire 2}, we may apply Lemma~\ref{lem: lemme intermediaire} with $m=n$ and $j_0=i_0,\dots,j_{m-1} = i_{n-1}$. This gives
    \begin{align*}
        \Phi_0(X_{i_0+1})\cdots \Phi_0(X_{i_{n-1}+1}) \leq X_{i_1}\cdots X_{i_{n-1}} \Phi_{n-1}(X_{i_{n-1}+1}),
    \end{align*}
    which together with \eqref{eq: inegalite de Nguyen} yields $\Phi_{n-1}(X_{i_{n-1}+1}) \geq c^{-1}$. Since the function $\Phi_{n-1}$ is monotonic, we deduce that $\Phi_{n-1}(X)\geq c^{-1}$ for each $X$ large enough, by letting $i_0$ go to infinity.
\end{proof}

\section{Proof of Theorem~\ref{thm: comportement points minimaux cas extremal}}
\label{section: preuves thm comportement points min}
    First, note that \eqref{eq: egalite expo dans thm points min} follows
    from Theorem~\ref{thm: transfert exposants avec constante}. So it only
    remains to prove the second part of
    Theorem~\ref{thm: comportement points minimaux cas extremal}.
    Let $(\ux_i)_{i\geq 0}$ be a sequence of minimal points for $\bxi$
    with respect to $\SSS$. For each $i\geq 0$, we write
    \[
     X_i = \norm{\ux_i} \et L_i=\LL(X_i;\SSS) = L_{\bxi}(\ux_i).
    \]
    The sequence $(\uy_i)_{i\geq 0}$ will be constructed as a subsequence
    of $(\ux_i)_i$ so that Condition~\ref{condition iv thm suite} of
    Theorem~\ref{thm: comportement points minimaux cas extremal} will be
    automatically satisfied. In this section, all implicit constants depend
    only on $\bxi,a,b,\alpha,\beta$. For each $X>0$, we set
    \begin{equation*}
        (\psi,\phi,\theta)(X)
        = \Big(bX^{-\beta},aX^{-\alpha},
               \Big(\frac{a}{b}\Big)^{-1/\beta}X^{\alpha/\beta} \Big),
    \end{equation*}
    as in \eqref{eq: def psi, phi, theta for power functions}.  Then,
    for $k=0,\dots,n-1$, we denote by $\phi_k$ and $\Phi_k$ the
    functions defined
    on $(0,\infty)$ by the formulas
    \eqref{eq: def phi_0}--\eqref{eq: def Phi_k}
    from Theorem~\ref{thm: thm transfert}.
    We also fix an index $\ell_0$ such that the main hypothesis
    \eqref{eq: encadrement initial version exposants} is satisfied for
    each $X\geq X_{\ell_0}$.

    Consider the sequence $i_0 < i_1 < \dots < i_{n-1}$ given
    by Theorem~\ref{thm: lemme Nguyen} for a choice of $i_0\geq \ell_0$.
    For each $k=0,\dots,n-1$, we set
    \begin{equation}
     \label{eq: preuve thm 2.1 choix y,Y,z,Z}
     (\uy_k,Y_k) = (\ux_{i_k}, X_{i_k})
     \et
     (\uz_k,Z_k) = (\ux_{i_k+1}, X_{i_k+1}).
    \end{equation}
    By construction, we have
    \begin{equation}
        \label{eq: preuve thm 2.1 inter 0}
        \langle\uy_0,\uz_0\rangle_\bR=\langle\uy_0,\uy_1\rangle_\bR
        \et
        \langle\uy_0,\dots,\uy_{n-1},\uz_{n-1}\rangle_\bR=\bR^{n+1}.
    \end{equation}
    Using \eqref{eq: eq inter triplet compatibles functions inter},
    we also find that
    \begin{align}
         \label{eq: preuve thm 2.1 inter 0 explicitation phi_k}
        \Phi_{k}(X) = X\phi_{k}(X) = c_{k}X^{\ee_k}\quad \textrm{with } \ee_k=1-\alpha-\dots-\frac{\alpha^{k+1}}{\beta^{k}},
    \end{align}
    for each $k=0,\dots,n-1$ and each $X>0$, where $c_{k} > 0$ depends
    only on $a,b,\alpha,\beta$. Note that
    \[
     \ee_0 > \cdots > \ee_{n-1} = \ee \geq 0
    \]
    where $\ee$ is given by \eqref{eq: egalite expo dans thm points min}.
    We find
    \begin{align*}
        c^{-1}\prod_{k = 1}^{n-1}Y_k &\leq \prod_{k = 0}^{n-1}Z_kL_{\bxi}(\uy_k)
          & & \textrm{by Theorem~\ref{thm: lemme Nguyen}}, \\
        & \leq \prod_{k = 0}^{n-1}\Phi_0(Z_k)
          & & \textrm{by \eqref{eq: X_i minore par X_i+1}}, \\
        & \leq \Big(\prod_{k = 1}^{n-1}Y_k\Big)\Phi_{n-1}(Z_{n-1})
          & & \textrm{by Lemma~\ref{lem: lemme intermediaire} with $m=n$},\\
        & = \Big(\prod_{k = 1}^{n-1}Y_k\Big)c_{n-1}Z_{n-1}^\ee
          & & \textrm{by \eqref{eq: preuve thm 2.1 inter 0 explicitation phi_k}}.
    \end{align*}
    This uses sequentially the inequalities
    \[
        L_{\bxi}(\uy_k)\leq \phi(Z_k) \quad (0\leq k < n),
    \]
    coming from \eqref{eq: X_i minore par X_i+1} as well as the
    inequalities
    \[
        \Phi_{k-1}(Z_{k-1})\leq \Phi_{k-1}(Y_k)
        \et
        \phi_{k-1}(Y_k) \leq \phi_{k-1}(\theta(Z_{k}))
        \quad (1\leq k < n)
    \]
    coming from \eqref{eq: preuve lemme tranfert fonctionnel} in the proof
    of Lemma~\ref{lem: lemme intermediaire} with $m=n$.
    In each of these inequalities the ratio of the right-hand side divided
    by the left-hand side is therefore at most $cc_{n-1}Z_{n-1}^{\ee}$.
    Using \eqref{eq: preuve thm 2.1 inter 0 explicitation phi_k} and the
    fact that for $k=1,\dots,n-1$ we have
    \[
        \ee_{k-1}
          = \ee + \frac{\alpha^{k+1}}{\beta^{k}} + \dots
                + \frac{\alpha^{n}}{\beta^{n-1}}
         \geq \alpha\Big(\frac{\alpha}{\beta}\Big)^{k}
        \et
        1-\ee_{k-1} \geq \alpha,
    \]
    we thus get the following estimates
    \begin{align}
        \label{eq: L(Y_k) proche Z_k}
        \big|\log L_{\bxi}(\uy_k)+ \alpha\log Z_k \big|
          & \leq  \GrO(1) + \ee \log Z_{n-1}
          &&(0\leq k < n),\\
        \label{eq: Y_k Z_k-1}
        \big|\log Y_k - \log Z_{k-1} \big|
          & \leq \GrO(1) + \frac{\ee}{\alpha}\Big(\frac{\beta}{\alpha}\Big)^{k} \log Z_{n-1}
          && (1\leq k < n),\\
        \label{eq: Y_k proche a/b*Z_k}
        \big|\log Y_k - \frac{\alpha}{\beta}\log Z_k \big|
          & \leq \GrO(1) + \frac{\ee}{\alpha} \log Z_{n-1}
          && (1\leq k < n).
    \end{align}

    Suppose from now on that $\epsilon$ satisfies the inequality
    \eqref{eq: preuve thm 2.1 def ee_0} of
    Theorem~\ref{thm: comportement points minimaux cas extremal}.
    We distinguish two cases.


   \subsection*{First case:} $\alpha<\beta$.
   We start by noting that
   \begin{align}
   \label{eq: preuve thm 2.1 inter 1.2 bis}
    \log Z_{n-1}
      \leq  \GrO(1) + 2\Big(\frac{\beta}{\alpha}\Big)^{n-k}\log Z_{k-1}
      \quad (1\leq k < n).
   \end{align}
   Indeed, \eqref{eq: Y_k Z_k-1} and \eqref{eq: Y_k proche a/b*Z_k}
   imply that
   \begin{equation*}
    \log Z_k
      \leq \GrO(1) + \frac{\beta}{\alpha}\log Z_{k-1}
           + \frac{2\ee}{\alpha}\Big(\frac{\beta}{\alpha}\Big)^{k+1} \log Z_{n-1}
    \quad (1\leq k < n),
   \end{equation*}
   and by descending induction starting with $k=n-1$, we obtain
   \begin{align*}
    \log Z_{n-1}
      \leq \GrO(1) + \Big(\frac{\beta}{\alpha}\Big)^{n-k}\log Z_{k-1}
           + \frac{2(n-k)\ee}{\alpha}\Big(\frac{\beta}{\alpha}\Big)^{n} \log Z_{n-1}
    \quad (1\leq k < n).
   \end{align*}
   This yields \eqref{eq: preuve thm 2.1 inter 1.2 bis} since
   by \eqref{eq: preuve thm 2.1 def ee_0} the coefficient of $\log Z_{n-1}$
   in the right-hand side is less than $1/2$.

   Combining \eqref{eq: Y_k proche a/b*Z_k} and
   \eqref{eq: preuve thm 2.1 inter 1.2 bis} together with
   $Z_{k-1}\leq Y_k$, we obtain
   \begin{equation}
    \label{eq: preuve thm 2.1 relation Z_k Y_k}
    \big|\alpha\log Z_k -\beta\log Y_k \big|
        \leq \GrO(1) + 2\ee\Big(\frac{\beta}{\alpha}\Big)^{n-k+1}\log Y_k
     \quad(1\leq k < n).
   \end{equation}
   Thus there exists a constant $C>0$ (depending only on
   $\xi,a,b,\alpha,\beta$) such that
   \begin{equation}
    \label{eq: preuve thm 2.1 inter 2.2}
    \Big|\log X_{i+1} -\frac{\beta}{\alpha}\log X_{i} \Big|
      \leq C + 2\frac{\ee}{\alpha}\Big(\frac{\beta}{\alpha}\Big)^{n}\log X_{i}
   \end{equation}
   for each $i$ among $\{i_1,i_2,\dots,i_{n-1}\}$.  By \eqref{eq: Y_k Z_k-1}
   and \eqref{eq: preuve thm 2.1 inter 1.2 bis}, we also have
   \begin{equation}
    \label{eq: preuve thm 2.1 inter 2.3}
     \big|\log Y_k- \log Z_{k-1} \big|
      \leq \GrO(1) + 2\frac{\ee}{\alpha}\Big(\frac{\beta}{\alpha}\Big)^{n} \log Z_{k-1}
     \quad (1\leq k < n).
   \end{equation}
   For the intermediate indices $i$ with $i_{k-1} < i < i_{k}$ for
   some $k\in\{1,\dots,n-1\}$, we have $Z_{k-1}\leq X_i < X_{i+1}
   \leq Y_k$, and the above estimate yields
   \begin{equation}
    \label{eq: preuve thm 2.1 inter 2.4}
    \big| \log X_{i+1} - \log X_i \big|
        \leq C + 2\frac{\ee}{\alpha}\Big(\frac{\beta}{\alpha}\Big)^{n} \log X_{i},
   \end{equation}
   at the expense of replacing $C$ by a larger constant if necessary.

   By the hypothesis \eqref{eq: preuve thm 2.1 def ee_0} on $\epsilon$
   and the fact that $\beta/\alpha > 1$,
   the inequalities \eqref{eq: preuve thm 2.1 inter 2.2} and
   \eqref{eq: preuve thm 2.1 inter 2.4} cannot hold simultaneously
   for any sufficiently large integer $i$, say for any $i\ge \ell_1$
   where $\ell_1\ge \ell_0$.  Define $I$ to be the set of all integers
   $i\geq \ell_1$ for which \eqref{eq: preuve thm 2.1 inter 2.2} holds.
   Then, for a sequence $i_0<i_1<\cdots<i_{n-1}$ as above, with
   $i_0\ge \ell_1$, we have $I\cap (i_0,i_{n-1}]=\{i_1,i_2,\dots,i_{n-1}\}$.
   In particular, the set $I$ is infinite and, if we choose $i_0\in I$,
   then $i_0,i_1,\dots,i_{n-1}$ are $n$ consecutive elements of $I$.

   Denote by $i_0 < i_1 < \cdots$ the elements of $I$ and define
   $\uy_k$, $Y_k$, $\uz_k$ and $Z_k$ by
   \eqref{eq: preuve thm 2.1 choix y,Y,z,Z} for each $k\ge 0$.
   By the above, the relations \eqref{eq: preuve thm 2.1 inter 0}
   extend to
   \[
     \langle\uy_k,\uz_k\rangle_\bR=\langle\uy_k,\uy_{k+1}\rangle_\bR
     \et
     \langle\uy_k,\dots,\uy_{k+n-1},\uz_{k+n-1}\rangle_\bR=\bR^{n+1}
   \]
   for each $k\ge 0$. Thus $\{\uy_k,\dots,\uy_{k+n-1},\uy_{k+n}\}$
   spans $\bR^{n+1}$ for each $k\ge 0$ and so $(\uy_k)_{k\ge 0}$
   satisfies Condition (iii) of the theorem.  Applying
   \eqref{eq: L(Y_k) proche Z_k},
   \eqref{eq: preuve thm 2.1 inter 1.2 bis},
   \eqref{eq: preuve thm 2.1 relation Z_k Y_k}
   and \eqref{eq: preuve thm 2.1 inter 2.3}
   with $k=n-1$ (which is possible since $n\ge 2$), we also obtain
   that
   \begin{align}
     \big|\log L_{\bxi}(\uy_k)+ \alpha\log Z_k \big|
      &\leq  \GrO(1) + \ee \log Z_k,
   \label{preuve1.4:eq5}\\
     \log Z_k
      &\leq  \GrO(1) + 2\Big(\frac{\beta}{\alpha}\Big)\log Z_{k-1},
   \label{preuve1.4:eq6}\\
     \big|\alpha\log Z_k -\beta\log Y_k \big|
      &\leq \GrO(1) + 2\ee\Big(\frac{\beta}{\alpha}\Big)^2\log Y_k,
   \label{preuve1.4:eq7}\\
     \big|\log Y_k- \log Z_{k-1} \big|
      &\leq \GrO(1) + 2\frac{\ee}{\alpha}\Big(\frac{\beta}{\alpha}\Big)^{n} \log Z_{k-1},
   \label{preuve1.4:eq8}
   \end{align}
   for each $k\ge n-1$.  Combining the first three inequalities
   \eqref{preuve1.4:eq5}--\eqref{preuve1.4:eq7}, we find
   \begin{align*}
    \big|\log L_{\bxi}(\uy_k) + \beta\log Y_k \big|
      &\leq  \GrO(1) + 2\ee \Big(\frac{\beta}{\alpha}\Big)\log Z_{k-1}
         + 2\ee\Big(\frac{\beta}{\alpha}\Big)^2\log Y_k \\
      &\leq  \GrO(1) + 4\ee\Big(\frac{\beta}{\alpha}\Big)^2\log Y_k
   \end{align*}
   since $Z_{k-1}\le Y_k$.  Thus Condition (ii) is fulfilled.
   Finally, replacing $k$ by $k+1$ in \eqref{preuve1.4:eq8} and
   using \eqref{preuve1.4:eq7}, we find
   \begin{align*}
    \big|\alpha\log Y_{k+1} -\beta\log Y_k \big|
     &\leq \GrO(1)
        + 2\ee\Big(\frac{\beta}{\alpha}\Big)^{n} \log Z_{k}
        + 2\ee\Big(\frac{\beta}{\alpha}\Big)^2\log Y_k \\
     &\leq \GrO(1)
        + 4\ee\Big(\frac{\beta}{\alpha}\Big)^{n} \log Y_{k+1}
   \end{align*}
   since $Y_k\le Z_k\le Y_{k+1}$.  Thus Condition (i) is satisfied
   as well.


    \subsection*{Second case:} $\alpha=\beta$.
    Then we have $\ee=0$ and $\alpha=\beta=1/n$. Moreover,
    the hypothesis \eqref{eq: encadrement initial version exposants}
    implies that
    \begin{equation}
      \label{eq:preuve1.4:cas2:eq1}
      L_{\bxi}(\ux_i) \asymp X_i^{-1/n}
      \quad (i\ge 0).
    \end{equation}
    Thus the estimate \eqref{eq: L(Y_k) proche Z_k} with $k=0$
    yields $Y_0\asymp Z_0$, while \eqref{eq: Y_k Z_k-1}
    and \eqref{eq: Y_k proche a/b*Z_k} simplify to
    \[
     Z_0 \asymp Y_1 \asymp Z_1 \asymp\dots \asymp Y_{n-1} \asymp Z_{n-1}.
    \]
    Thus $\{\ux_{i_0},\ux_{i_1},\dots,\ux_{i_{n-1}},\ux_{i_{n-1}+1}\}$
    is a basis of $\bR^{n+1}$ with
    \begin{equation}
      \label{eq:preuve1.4:cas2:eq2}
      \norm{\ux_{i_0}}\asymp
      \norm{\ux_{i_1}}\asymp\cdots\asymp \norm{\ux_{i_{n-1}}}
      \asymp \norm{\ux_{i_{n-1}+1}}.
    \end{equation}

    We now construct recursively a subsequence $(\uy_k)_{k\geq 0}$
    of $(\ux_i)_{i\geq 0}$ such that
    \[
        \norm{\uy_k} \asymp \norm{\uy_{k+1}}
        \et
        \langle\uy_k,\dots,\uy_{k+n}\rangle_\bR=\bR^{n+1}
    \]
    for each $k\geq 0$. To start, we simply choose $i_0=\ell_0$ and set
    $(\uy_0,\dots,\uy_{n}) = (\ux_{i_0},\dots,\ux_{i_{n-1}},\ux_{i_{n-1}+1})$.
    Now suppose that $\uy_0,\dots,\uy_{k}$ have been constructed for
    an index $k\geq n$. Then $W=\langle \uy_{k-n+1},\dots,\uy_{k}\rangle_\bR$
    is a subspace of $\bR^{n+1}$ of dimension $n$.  We take $i_0$
    to be the index for which $\uy_{k} = \ux_{i_0}$.  By the above
    there exists a point $\uy_{k+1}$ among
    $\ux_{i_1},\dots,\ux_{i_{n-1}},\ux_{i_{n-1}+1}$ which lies
    outside of $W$.  Then $\{\uy_{k-n+1},\dots,\uy_{k+1}\}$ spans
    $\bR^{n+1}$, and by \eqref{eq:preuve1.4:cas2:eq2} we have
    $\norm{\uy_{k+1}} \asymp \norm{\uy_k}$.

    This sequence $(\uy_k)_{k\geq0}$ has all the requested properties
    since it also satisfies
    $L_{\bxi}(\uy_k) \asymp \norm{\uy_k}^{-1/n}$ for each $k\ge 0$
    by \eqref{eq:preuve1.4:cas2:eq1}.

\section{Applications}
\label{Section: applications}

The following result is implicit in the thesis of the first author. It follows from the proof of Theorem~2.1.3 of \cite{N2014} although the theorem by itself is a weaker assertion. We give a short proof based on Theorem~\ref{thm: transfert exposants avec constante}.

\begin{theorem}
  \label{thm: cor thm 1.2}
  Let $\theta$ be a real algebraic number of degree $n\geq 2$ and let $\xi\in\bR\setminus\bQ(\theta)$. Then the point $\bxi=(1,\theta,\dots,\theta^{n-1},\xi)\in\bR^{n+1}$ satisfies
  \begin{equation}
    \label{eq 1 thm-cor Van}
    \hlambda(\bxi) \leq \lambda_n
  \end{equation}
  where $\lambda_n$ is the unique positive solution of
  \begin{equation*}
    x+(n-1)x^2+\dots+(n-1)^{n-1}x^n = 1.
  \end{equation*}
  Moreover precisely, we have
  \begin{equation}
  \label{eq 2 thm-cor Van}
    \limsup_{X\rightarrow\infty}X^{\lambda_n}\LL(X) > 0.
  \end{equation}
\end{theorem}

\begin{proof}
  By Liouville's inequality, there exists a constant $c_1=c_1(\theta)>0$ such that the system
    \begin{align*}
    \max_{1\leq k\leq n-1} |y_k|\leq X^{1/(n-1)} \quad\textrm{and}\quad |y_0+\theta y_1+\dots + \theta^{n-1}y_{n-1}| \leq c_1X^{-1}
  \end{align*}
  admits no non-zero integer solution $(y_0,\dots,y_{n-1})$ for any $X\geq 1$. By Khinthine's transference principle \cite[Theorem~5A]{Schmidt1980}, there is therefore a constant $c_2=c_2(\theta)>0$ such that the dual system
  \begin{align}
    \label{eq 2 bis thm-cor Van proof}
    |x_0|\leq X \quad\textrm{and}\quad \max_{1\leq k \leq n-1}|x_k-\theta^k x_0| \leq c_2 X^{-1/(n-1)}
  \end{align}
  admits no non-zero integer solution $(x_0,\dots,x_{n-1})$ for each $X\geq 1$. Thus, we have
  \begin{equation}
   \label{eq 0 thm-cor Van proof}
    c_2X^{-1/(n-1)} \leq \LL(X)
  \end{equation}
  for each $X\geq 1$. If $\LL(X)\geq X^{-\lambda_n}$ for arbitrarily large values of $X$, then \eqref{eq 2 thm-cor Van} is immediate. Otherwise, Condition~\eqref{eq: encadrement initial version exposants thm1.2} of Theorem~\ref{thm: transfert exposants avec constante} is fulfilled with $\alpha = \lambda_n$ and $\beta=1/(n-1)$. As this yields an equality in \eqref{eq: thm 2 egalite ratio exposants}, we again get \eqref{eq 2 thm-cor Van} as a consequence of \eqref{eq: thm 1.2 estimation des constantes}.
\end{proof}

In the case $n=2$, the number $\lambda_2\cong 0.618$ is the inverse of the golden ratio and it follows from \cite{roy2013conics}~-- which more generally deals with approximation to real points on conics in $\bP^2(\bR)$-- that the upper bound \eqref{eq 1 thm-cor Van} is best possible: for any quadratic number $\theta\in\bR\setminus\bQ$, there exists $\xi\in\bR\setminus\bQ(\theta)$ such that $\bxi=(1,\theta,\xi)$ satisfies $\limsup X^{\lambda_2}\LL(X) <\infty$ and $\hlambda(\bxi)=\lambda_2$. For $n\geq 3$ the optimal upper bound is not known.

In \cite{poelsroy2019}, the second and the third authors apply Theorems~\ref{thm: transfert exposants avec constante} and~\ref{thm: comportement points minimaux cas extremal} to extend the results of \cite{kleinbock2018simultaneous} and \cite{roy2013conics} to points on general quadratic hypersurfaces of $\bP^n(\bR)$ defined over $\bQ$.

\bibliographystyle{abbrv}

\end{document}